\documentclass[12pt]{amsart}
\usepackage[active]{srcltx}
\usepackage{calc,amssymb,amsthm,amsmath,amscd, eucal,ulem}
\usepackage{alltt}
\usepackage[left=1.35in,top=1.25in,right=1.35in,bottom=1.25in]{geometry}
\RequirePackage[dvipsnames,usenames]{color}

\input{kmacros3.sty}
\input{xy}
\xyoption{all}
\usepackage{tikz}

\numberwithin{equation}{theorem}

\newcommand{\D}{\displaystyle}

\DeclareMathOperator{\soc}{Soc}

\renewcommand{\m}{\mathfrak{m}}

\DeclareMathOperator{\Ass}{Ass}

\usepackage{fullpage}

\usepackage{setspace}
\usepackage{hyperref}

\usepackage{enumerate}

\usepackage{graphicx}

\usepackage[all,cmtip]{xy}
\usepackage{verbatim}

\theoremstyle{theorem}

\begin{document}
\title{The category of $F$-modules has finite global dimension}
\author{Linquan Ma}
\address{Department of Mathematics\\ University of Michigan\\ Ann Arbor\\ Michigan 48109}
\email{lquanma@umich.edu}
\maketitle

\begin{abstract}
Let $R$ be a regular ring of characteristic $p>0$. In \cite{HochsterfinitenesspropertyofLyubeznikFmodule}, Hochster showed that the category of Lyubeznik's $F_R$-modules has enough injectives, so that every $F_R$-module has an injective resolution in this category. We show in this paper that under mild conditions on $R$, for example when $R$ is essentially of finite type over an $F$-finite regular local ring, the category of $F$-modules has finite global dimension $d+1$ where $d=\dim R$. In \cite{HochsterfinitenesspropertyofLyubeznikFmodule}, Hochster also showed that when $M$ and $N$ are $F_R$-finite $F_R$-modules, $\Hom_{F_R}(M,N)$ is finite. We show that in general $\Ext_{F_R}^1(M,N)$ is not necessarily finite.
\end{abstract}

\section{Introduction}
In \cite{HochsterfinitenesspropertyofLyubeznikFmodule}, Hochster showed some properties of Lyubeznik's $F$-modules:
\begin{theorem}[{\it cf.} Theorem 3.1 in \cite{HochsterfinitenesspropertyofLyubeznikFmodule}]
\label{Hochster existence of enough injectives}
The category of $F_R$-modules over a Noetherian regular ring $R$ of prime characteristic $p>0$ has enough injectives, i.e., every $F_R$-module can be embedded in an injective $F_R$-module.
\end{theorem}
\begin{theorem}[{\it cf.} Theorem 5.1 and Corollary 5.2(b) in \cite{HochsterfinitenesspropertyofLyubeznikFmodule}]
\label{Hochster finiteness of Hom(M,N)}
Let $R$ be a Noetherian regular ring of prime characteristic $p>0$. Let $M$ and $N$ be $F_R$-finite $F_R$-modules. Then $\Hom_{F_R}(M,N)$ is a finite-dimensional vector space over $\mathbb{Z}/p\mathbb{Z}$ and, hence, is a finite set. Moreover, when $R$ is local, every $F_R$-finite $F_R$-module has only finitely many $F_R$-submodules.
\end{theorem}

The main purpose of this paper is to get some further results based on Hochster's results. In connection with Theorem \ref{Hochster existence of enough injectives}, we prove the following (this can be viewed as an analogue of the corresponding statement for $\scr{D}$-modules in characteristic $0$):
\begin{theorem}
\label{main theorem}
Let $R$ be an $F$-finite regular ring of characteristic $p>0$ such that there exists a canonical module $\omega_R$ with $F^{!}\omega_R\cong \omega_R$ (this holds if $R$ is essentially of finite type over an $F$-finite regular local ring). Then the category of $F_R$-modules has finite global dimension $d+1$ where $d=\dim R$.
\end{theorem}

Theorem \ref{Hochster finiteness of Hom(M,N)} makes it quite natural to ask whether the higher $\Ext$ groups are also finite in this category (when $M$ and $N$ are $F_R$-finite $F_R$-modules). We show that in general this fails even for $\Ext^1$:

\begin{example}
Let $(R,\m, K)$ be a regular local ring of characteristic $p>0$ and dimension $d\geq 1$, and let $E=E(R/\m)$ be the injective hull of the residue field. Then $\Ext_{F_R}^1(R, E)\neq 0$. Moreover, when $K$ is infinite, $\Ext_{F_R}^1(R, E)$ is also infinite. In particular, $E$ is not injective in the category of $F_R$-modules.
\end{example}

This paper is organized as follows. In Section $2$ we review the definitions and basic properties of right $R\{F\}$-modules (i.e., Cartier modules) and Lyubeznik's $F_R$-modules, and we introduce the notion of {\it unit right $R\{F\}$-modules} which is motivated by the ideas in \cite{EmertonKisinRiemannHilbertcorrespondenceforunitFcrystals} and \cite{BlickleBoeckleCartierModulesFiniteness}. In Section $3$ we prove Theorem \ref{main theorem}, and we also obtain some results of independent interest on right $R\{F\}$-modules and unit right $R\{F\}$-modules. In Section $4$ we show some (non)finiteness results on $\Ext_{F_R}^1(M,N)$ when $M$ and $N$ are $F_R$-finite $F_R$-modules. Examples will be given throughout.

\section{Preliminaries}

Throughout this paper, $R$ will always denote a Noetherian regular ring of characteristic $p>0$ and dimension $d$. We use $R^{(e)}$ to denote the target ring of the $e$-th Frobenius map $F^e$: $R\rightarrow R$. When $M$ is an $R$-module and $x\in M$ is an element, we use $M^{(e)}$ to denote the corresponding module over $R^{(e)}$ and $x^{(e)}$ to denote the corresponding element in $M^{(e)}$. We shall let $F^e(-)$ denote the Peskine-Szpiro's Frobenius functor from $R$-modules to $R$-modules. In detail, $F^e(M)$ is given by base change to $R^{(e)}$ and then identifying $R^{(e)}$ with $R$. Note that by Kunz's result \cite{KunzCharacterizationsOfRegularLocalRings}, we know that $R^{(e)}$ is faithfully flat as an $R$-module. We say $R$ is {\it $F$-finite} if $R^{(1)}$ is finitely generated as an $R$-module. So for an $F$-finite regular ring, $R^{(1)}$ (and hence $R^{(e)}$ for every $e$) is finite and projective as an $R$-module.

We use $R\{F\}$ to denote the Frobenius skew polynomial ring, which is the noncommutative ring generated over $R$ by the symbols $1,F,F^2,\dots$ by requiring that $Fr=r^pF$ for $r\in R$. Note that $R\{F\}$ is always free as a left $R$-module and flat as a right $R$-module. When $R$ is $F$-finite, $R\{F\}$ is projective as a right $R$-module (because $R^{(1)}$ is projective in this case). We say an $R$-module $M$ is a {\it right $R\{F\}$-module} if it is a right module over the ring $R\{F\}$, or equivalently, there exists a morphism $\phi$: $M\rightarrow M$ such that for all $r\in R$ and $x\in M$, $\phi(r^px)=r\phi(x)$ (the right action of $F$ can be identified with $\phi$). This morphism can be also viewed as an $R$-linear map $\phi$: $M^{(1)}\rightarrow M$. We note that a right $R\{F\}$-module is the same as a {\it Cartier module} defined in \cite{BlickleBoeckleCartierModulesFiniteness} (where it is defined for general Noetherian rings and schemes of characteristic $p>0$).


We collect some definitions from \cite{LyubeznikFModulesApplicationsToLocalCohomology}. These are the main objects that we shall study in this paper. \begin{definition}[{\it cf.} Definition 1.1 in \cite{LyubeznikFModulesApplicationsToLocalCohomology}]
An {\it $F_R$-module} is an $R$-module $M$ equipped with an $R$-linear isomorphism $\theta$: $M\rightarrow F(M)$ which we call the structure morphism of $M$. A homomorphism of $F_R$-modules is an $R$-module homomorphism $f$: $M\rightarrow M'$ such that the following diagram commutes
\[\xymatrix{
M\ar[r]^{f}\ar[d]^{\theta} & M'\ar[d]^{\theta'}\\
F(M)\ar[r]^{F(f)}&F(M')\\
}\]
\end{definition}
\begin{definition}[{\it cf.} Definition 1.9  and Definition 2.1 in \cite{LyubeznikFModulesApplicationsToLocalCohomology}]
\label{generating morphism for F-mod}
A generating morphism of an $F_R$-module $M$ is an $R$-module homomorphism $\beta$: $M_0\rightarrow F(M_0)$, where $M_0$ is some $R$-module, such that $M$ is the limit of the inductive system in the top row of the commutative diagram
\[\xymatrix{
M_0\ar[r]^{\beta}\ar[d]^{\beta} & F(M_0)\ar[d]^{F(\beta)}\ar[r]^{F(\beta)}& F^2(M_0)\ar[r]^{F^2(\beta)}\ar[d]^{F^2(\beta)}&{\cdots}\\
F(M_0)\ar[r]^{F(\beta)}& F^2(M_0)\ar[r]^{F^2(\beta)}&F^3(M_0)\ar[r]^{F^3(\beta)}&{\cdots}\\
}\]
and $\theta$: $M\rightarrow F(M)$, the structure isomorphism of $M$, is induced by the vertical arrows in this diagram. An $F_R$-module $M$ is called {\it $F_R$-finite} if $M$ has a generating morphism $\beta$: $M_0\rightarrow F(M_0)$ with $M_0$ a finitely generated $R$-module.
\end{definition}

Now we introduce the notion of {\it unit right $R\{F\}$-modules} which is an analogue of {\it unit left $R\{F\}$-modules} in \cite{EmertonKisinRiemannHilbertcorrespondenceforunitFcrystals}. This is a key concept in relating Lyubeznik's $F_R$-modules with right $R\{F\}$-modules. The ideas can be also found in Section 5.2 in \cite{BlickleBoeckleCartierModulesFiniteness}. We first recall the functor $F^!(-)$ in the case that $R$ is regular and $F$-finite: for any $R$-module $M$, $F^!(M)$ is the $R$-module obtained by first considering $\Hom_R(R^{(1)}, M)$ as an $R^{(1)}$-module and then identifying $R^{(1)}$ with $R$. Remember that giving an $R$-module $M$ a right $R\{F\}$-module structure is equivalent to giving an $R$-linear map $M^{(1)}\rightarrow M$. But this is the same as giving an $R^{(1)}$-linear map $M^{(1)}\rightarrow\Hom_R(R^{(1)},M)$. Hence after identifying $R^{(1)}$ with $R$, we find that giving $M$ a right $R\{F\}$-module structure is equivalent to giving a map $\tau$: $M\rightarrow F^!M$. Moreover, it is straightforward to check that a homomorphism of right $R\{F\}$-modules is an $R$-module homomorphism $g$: $M\rightarrow M'$ such that the following diagram commutes
\[\xymatrix{
M\ar[r]^{g}\ar[d]^{\tau} & M'\ar[d]^{\tau'}\\
F^!M\ar[r]^{F^!(g)}&F^!M'\\
}\]

\begin{definition}
A unit right $R\{F\}$-module is a right $R\{F\}$-module $M$ such that the structure map $\tau$: $M\rightarrow F^!M$ is an isomorphism.
\end{definition}

\begin{remark}
\label{generating morphism for unit right R[F]-mod}
Similar to Definition \ref{generating morphism for F-mod}, we introduce the notion of generating morphism of unit right $R\{F\}$-modules. Let $M_0$ be a right $R\{F\}$-module with structure morphism $\tau_0$: $M_0\rightarrow F^!(M_0)$. Let $M$ be the limit of the inductive system in the top row of the commutative diagram
\[\xymatrix{
M_0\ar[r]^{\tau_0}\ar[d]^{\tau_0} & F^!(M_0)\ar[d]^{F^!(\tau_0)}\ar[r]^{F^!(\tau_0)}& (F^!)^2(M_0)\ar[r]^-{(F^!)^2(\tau_0)}\ar[d]^{(F^!)^2(\tau_0)}&{\cdots}\\
F^!(M_0)\ar[r]^{F^!(\tau_0)}& (F^!)^2(M_0)\ar[r]^{(F^!)^2(\tau_0)}&(F^!)^3(M_0)\ar[r]^-{(F^!)^3(\tau_0)}&{\cdots}\\
}\]
Since $R$ is $F$-finite, it is easy to see that $F^!(-)$ commutes with direct limit. Hence $\tau$: $M\rightarrow F^!M$ induced by the vertical arrows in the above diagram is an isomorphism. $M$ is a unit right $R\{F\}$-module.
\end{remark}

For an $F$-finite regular ring $R$, any rank $1$ projective module is a canonical module $\omega_R$ of $R$ (we refer to \cite{HartshorneResidues} for a detailed definition of canonical module and dualizing complex). When $R$ is local, $\omega_R=R$ is unique. It is easy to see that $F^{!}\omega_R$ is always a canonical module of $R$ (see \cite{HartshorneResidues} for more general results). However, to the best of our knowledge, it is still unknown whether there always exists $\omega_R$ such that $F^{!}\omega_R\cong\omega_R$ for $F$-finite regular ring $R$. Nonetheless, it is true if either $R$ is essentially of finite type over an $F$-finite regular local ring or $R$ is sufficiently affine. We refer to Proposition 2.20 and Proposition 2.21 in \cite{BlickleBoeckleCartierModulesFiniteness} as well as \cite{HartshorneResidues} for more details on this question.

The next theorem is well known. It follows from duality theory in \cite{HartshorneResidues}. In the context of the Frobenius morphism it is explained in \cite{BlickleBoeckleCartierModulesFiniteness}. Since we need to use this repeatedly throughout the article, we give a short proof for completeness.
\begin{theorem}
\label{equivalence of category of F-mod and unit right mod}
Let $R$ be an $F$-finite regular ring such that there exists a canonical module $\omega_R$ with $F^!\omega_R\cong\omega_R$. Then the category of unit right $R\{F\}$-modules is equivalent to the category of $F_R$-modules. Moreover, the equivalence is given by tensoring with $\omega_R^{-1}$, its inverse by tensoring with $\omega_R$.
\end{theorem}
\begin{proof}
We first note that, for any $R$-module $M$, \[(\omega_R^{-1})^{(1)}\otimes_{R^{(1)}}\Hom_R(R^{(1)}, M)\cong (\omega_R^{-1})^{(1)}\otimes_{R^{(1)}}\Hom_R(R^{(1)}, \omega_R)\otimes_R(\omega_R^{-1}\otimes_RM).\] Hence after identifying $R^{(1)}$ with $R$, the above equality becomes \[\omega_R^{-1}\otimes_RF^!M\cong \omega_R^{-1}\otimes_RF^!\omega_R\otimes_RF(\omega_R^{-1}\otimes_RM)\cong F(\omega_R^{-1}\otimes_RM)\] where the last equality is by our assumption $F^!\omega_R\cong\omega_R$. Now for any unit right $R\{F\}$-module $M$, we have an isomorphism $M\xrightarrow{\tau} F^!M$. Hence after tensoring with $\omega_R^{-1}$, we get $\omega^{-1}\otimes_RM\xrightarrow{id\otimes_R\tau} \omega^{-1}\otimes_RF^!M\cong F(\omega_R^{-1}\otimes_RM)$. This shows that $\omega_R^{-1}\otimes_RM$ is an $F_R$-module with structure morphism $\theta$ given by $id\otimes_R\tau$. The converse can be proved similarly.
\end{proof}

Throughout the rest of the paper, we will use $\Ext_R^i$, $\Ext_{R\{F\}}^i$, $\Ext_{uR\{F\}}^i$, and $\Ext_{F_R}^i$ (respectively, $\id_R$, $\id_{R\{F\}}$, $\id_{uR\{F\}}$, $\id_{F_R}$) to denote the $i$-th $\Ext$ group (respectively, the injective dimension) computed in the category of $R$-modules, right $R\{F\}$-modules, unit right $R\{F\}$-modules, and $F_R$-modules.

We end this section by studying some examples of $F_R$-modules. The simplest example of an $F_R$-module is $R$ equipped with structure isomorphism the identity map, that is, sending $1$ in $R$ to $1$ in $F(R)\cong R$. Note that this corresponds to the unit right $R\{F\}$-module $\omega_R\cong F^!\omega_R$ under Theorem \ref{equivalence of category of F-mod and unit right mod}. Another important example is $E=E(R/\m)$, the injective hull of $R/\m$ for a maximal ideal $\m$ of $R$. We can give it a generating morphism $\beta$: $R/\m\rightarrow F(R/\m)$ by sending $\overline{1}$ to $\overline{x_1^{p-1}\cdots x_d^{p-1}}$ (where $x_1,\dots,x_d$ represents minimal generators of $\m R_\m$). We will call these structure isomorphisms of $R$ and $E$ the {\it standard} $F_R$-module structures on $R$ and $E$. Note that in particular $R$ and $E$ with the standard $F_R$-module structures are $F_R$-finite $F_R$-modules. Now we provide a nontrivial example of an $F_R$-module:


\begin{example}
\label{example of F_R-module by shift}
Let $R^\infty:=\D\oplus_{i\in\mathbb{Z}}Rz_i$ denote the infinite direct sum of copies of $R$ equipped with the $F_R$-module structure by setting \[\theta: z_i\rightarrow z_{i+1}.\] Then $R^\infty$ is {\it not} $F_R$-finite. It is easy to see that we have a short exact sequence of $F_R$-modules: \[0\to R^\infty\xrightarrow{z_i\mapsto z_i-z_{i+1}}R^\infty\xrightarrow{z_i\mapsto 1}R\to 0\]
where the last $R$ is equipped with the standard $F_R$-module structure.

We want to point out that the above sequence does not split in the category of $F_R$-modules. Suppose $g$: $R\rightarrow R^\infty$ is a splitting, say $g(1)=\{y_j\}_{j\in\mathbb{Z}}\neq 0$. Then a direct computation shows that $\theta(\{y_j\})=\{y_j^p\}$, which is impossible by the definition of $\theta$. Hence, by Yoneda's characterization of $\Ext$ groups, we know that $\Ext_{F_R}^1(R,R^\infty)\neq 0$.
\end{example}

\section{The global dimension of Lyubeznik's $F$-modules}
Our goal in this section is to prove Theorem \ref{main theorem}. First we want to show that, when $R$ is $F$-finite, the category of right $R\{F\}$-modules has finite global dimension $d+1$. We start with a lemma which is an analogue of Lemma 1.8.1 in \cite{EmertonKisinRiemannHilbertcorrespondenceforunitFcrystals}.
\begin{lemma}
\label{two-step resolution of right R[F]-module}
Let $R$ be a regular ring and let $M$ be a right $R\{F\}$-module, so that there is an $R$-linear map $\phi$: $M^{(1)}\rightarrow M$ (so for every $i$, we get an $R$-linear map $\phi^i$: $M^{(i)}\rightarrow M$ by composing $\phi$ $i$ times). Then we have an exact sequence of right $R\{F\}$-modules
\begin{equation*}
0\rightarrow M^{(1)}\otimes_RR\{F\}\xrightarrow{\alpha} M\otimes_RR\{F\}\xrightarrow{\beta} M\rightarrow 0
\end{equation*}
where for every $x^{(1)}\in M^{(1)}$,
\[\alpha(x^{(1)}\otimes F^i)=\phi(x^{(1)})\otimes F^i-x\otimes F^{i+1}\]
and for every $y\in M$,
\[\beta(y\otimes F^i)=\phi^{i}(y^{(i)}).\]
\end{lemma}
\begin{proof}
It is clear that every element in $M^{(1)}\otimes_RR\{F\}$ (resp. $M\otimes_RR\{F\}$) can be written uniquely as a finite sum $\sum x_i^{(1)}\otimes F^i$ where $x_i^{(1)}\in M^{(1)}$ (resp. $x_i\in M$) because $R\{F\}$ is free as a left $R$-module (this verifies that our maps $\alpha$ and $\beta$ are well-defined). It is straightforward to check that $\alpha$, $\beta$ are morphisms of right $R\{F\}$-modules and that $\beta\circ\alpha=0$ and $\beta$ is surjective (because $\beta(y\otimes 1)=\phi^0(y)=y$). So it suffices to show $\alpha$ is injective and $\ker(\beta)\subseteq \im (\alpha)$.

Suppose $\alpha(\sum x_i^{(1)}\otimes F^i)=0$. By definition of $\alpha$ we get $\sum(\phi(x_i^{(1)})-x_{i-1})\otimes F^i=0$. Hence by uniqueness we get $\phi(x_i^{(1)})=x_{i-1}$ for all $i$. Hence $x_i=0$ for all $i$ (because it is a finite sum). This proves $\alpha$ is injective.

Now suppose $\beta(\sum_{i=0}^n y_i\otimes F^i)=0$. We want to find $x_i$ such that
\begin{equation}
\label{solve equation to check exactness of two-step resolution}
\alpha(\sum_{i=0}^n x_i^{(1)}\otimes F^i)=\sum_{i=0}^n y_i\otimes F^i.
\end{equation}
By definition of $\beta$ we know that $\sum_{i=0}^n\phi^i(y_i^{(i)})=0$. Now one can check that
\begin{eqnarray*}
&&x_0=-(y_1+\phi(y_2^{(1)})+\cdots+\phi^{n-1}(y_n^{(n-1)}))\\
&&x_2=-(y_2+\phi(y_3^{(1)})+\cdots+\phi^{n-2}(y_n^{(n-2)}))\\
&&\cdots\\
&&x_{n-1}=-y_n\\
&&x_n=0
\end{eqnarray*}
is a solution of (\ref{solve equation to check exactness of two-step resolution}). This proves $\ker(\beta)\subseteq \im (\alpha)$.
\end{proof}

In \cite{EmertonKisinRiemannHilbertcorrespondenceforunitFcrystals}, a similar two-step resolution is proved for left $R\{F\}$-modules (see Lemma 1.8.1 in \cite{EmertonKisinRiemannHilbertcorrespondenceforunitFcrystals}). And using the two-step resolution it is proved in \cite{EmertonKisinRiemannHilbertcorrespondenceforunitFcrystals} that the category of left $R\{F\}$-modules has $\Tor$-dimension at most $d+1$ (see Corollary 1.8.4 in \cite{EmertonKisinRiemannHilbertcorrespondenceforunitFcrystals}). We want to mimic the strategy and prove the corresponding results for right $R\{F\}$-modules. And we can actually improve the result: we show that when $R$ is $F$-finite, the category of right $R\{F\}$-modules has finite {\it global} dimension {\it exactly} $d+1$.
\begin{theorem}
\label{right R[F]-module has finite global dimension}
Let $R$ be an $F$-finite regular ring of dimension $d$. Then the category of right $R\{F\}$-modules has finite global dimension $d+1$.
\end{theorem}
\begin{proof}
We first note that for every right $R\{F\}$-module $M$ with structure map $\tau$: $M\to F^!M$, a projective resolution of $M$ {\it in the category of $R$-modules} can be given a structure of right $R\{F\}$-modules such that it becomes an exact sequence of right $R\{F\}$-modules. This is because we can lift the natural map $\tau$: $M\rightarrow F^!M$ to a commutative diagram
\[\xymatrixcolsep{1pc}\xymatrix{
0 \ar[r]  & P_k \ar[r] \ar@{.>}[d] & P_{k-1} \ar[r]  \ar@{.>}[d] &{\cdots} \ar[r] & P_1 \ar[r]  \ar@{.>}[d] & P_0 \ar[r]\ar@{.>}[d] & M \ar[r] \ar[d]^\tau    & 0\\
0 \ar[r] & F^!(P_k)  \ar[r] & F^!(P_{k-1}) \ar[r] &{\cdots} \ar[r] & F^!(P_1) \ar[r] & F^!(P_0) \ar[r] & F^!(M) \ar[r]  &0\\
}\]
because we can always lift a map from a complex of projective modules to an acyclic complex ($F^!(-)$ is an exact functor when $R$ is $F$-finite).

By Lemma \ref{two-step resolution of right R[F]-module}, we have an exact sequence of right $R\{F\}$-modules
\begin{equation}
\label{an explicit two step resolution of M}
0\rightarrow M^{(1)}\otimes_RR\{F\}\xrightarrow{\alpha} M\otimes_RR\{F\}\xrightarrow{\beta} M\rightarrow 0.
\end{equation}

Now we tensor the above (\ref{an explicit two step resolution of M}) with the projective resolution of $M$ over $R$, we have the following commutative diagram
\begin{equation}
\label{huge commutative diagram}
\xymatrixcolsep{1pc}\xymatrix{
0 \ar[r]  & P_{k} \ar[r]  & P_{k-1} \ar[r]   &{\cdots} \ar[r] & P_{1} \ar[r]   & P_0  \ar[r]    & 0\\
0 \ar[r] & P_{k}\otimes_RR\{F\}  \ar[r] \ar[u]& P_{k-1}\otimes_RR\{F\} \ar[r] \ar[u] &{\cdots} \ar[r] & P_{1}\otimes_RR\{F\} \ar[r] \ar[u]& P_{0}\otimes_RR\{F\} \ar[r] \ar[u] &0\\
0 \ar[r] & P_{k}^{(1)}\otimes_RR\{F\}  \ar[r] \ar[u]^{\alpha_k} & P_{k-1}^{(1)}\otimes_RR\{F\} \ar[r] \ar[u]^{\alpha_{k-1}} &{\cdots} \ar[r] & P_{1}^{(1)}\otimes_RR\{F\} \ar[r] \ar[u]^{\alpha_1} & P_{0}^{(1)}\otimes_RR\{F\} \ar[r] \ar[u]^{\alpha_0} &0\\
}\end{equation}

The first line is a projective resolution of $M$ over $R$. And by the above discussion we can give each $P_i$ a right $R\{F\}$-module structure such that it is an exact sequence of right $R\{F\}$-modules. The second line (resp. the third line) is obtained from the first line by tensoring with $R\{F\}$ (resp. applying $^{(1)}$ and then tensoring with $R\{F\}$). Each column is the map described in Lemma \ref{two-step resolution of right R[F]-module}. In particular, all columns are exact sequences of right $R\{F\}$-modules.

Let $C_\bullet$ be the complex of the third line and $D_\bullet$ be the complex of the second line of (\ref{huge commutative diagram}). The homology of the mapping cone of $C_\bullet\rightarrow D_\bullet$ is the same as the homology of the quotient complex $D_\bullet /C_\bullet$, which is the first line in (\ref{huge commutative diagram}). Hence the mapping cone is acyclic. Since each $P_i$ is projective as an $R$-module, we know that $P_i$ is a direct summand of a free $R$-module $G$. So $P_i^{(1)}$ is a direct summand of $G^{(1)}$. Since $R$ is $F$-finite, $G^{(1)}$ is projective as an $R$-module, so $P_i^{(1)}$ is also projective as an $R$-module. Hence $P_i^{(1)}\otimes_RR\{F\}$ and $P_i\otimes_RR\{F\}$ are projective as right $R\{F\}$-modules for every $i$. Thus the mapping cone of $C_\bullet \to D_\bullet$ gives a right $R\{F\}$-projective resolution of $M$. We note that this resolution has length $k+1$. Since we can always take a projective resolution of $M$ of length $k\leq d$, the right $R\{F\}$-projective resolution we obtained has length $\leq d+1$.

We have already seen that the global dimension of right $R\{F\}$-modules is $\leq d+1$. Now we let $M$ and $N$ be two right $R\{F\}$-modules with trivial right $F$-action (i.e., the structure maps of $M$ and $N$ are the zero maps). I claim that in this case, we have
\begin{equation}
\label{relating Ext_R[F] with Ext_R}
\Ext^{j}_{R\{F\}}(M, N)=\Ext^{j}_R(M, N)\oplus\Ext_R^{j-1}(M^{(1)}, N).
\end{equation}
To see this, we look at (\ref{huge commutative diagram}) applied to $M$ with trivial right $F$-action. It is clear that in this case each $P_i$ in the first line of (\ref{huge commutative diagram}) also has trivial right $R\{F\}$-module structure. So as described in Lemma \ref{two-step resolution of right R[F]-module}, we have
\begin{equation}
\label{an explicit two step resolution of a trivial right R[F]-module}
\alpha_j(x^{(1)}\otimes F^i)=-x\otimes F^{i+1}
\end{equation}
for every $x^{(1)}\otimes F^i\in P_j^{(1)}\otimes_RR\{F\}$. The key observation is that, since $N$ has trivial right $F$-action, when we apply $\Hom_{R\{F\}}(-, N)$ to $\alpha_j$: $P_j^{(1)}\otimes_RR\{F\}\rightarrow P_j\otimes_RR\{F\}$, the dual map $\alpha_j^\vee$ is the zero map (one can check this by a direct computation using (\ref{an explicit two step resolution of a trivial right R[F]-module})). Hence when we apply $\Hom_{R\{F\}}(-, N)$ to the mapping cone of $C_\bullet\to D_\bullet$, the $j$-th cohomology is the same as the direct sum of the $j$-th cohomology of $\Hom_{R\{F\}}(C_\bullet[-1], N)$ and the $j$-th cohomology of $\Hom_{R\{F\}}(D_\bullet, N)$. That is,
\begin{equation}
\label{Ext^j in terms of cohomology of D and C}
\Ext^{j}_{R\{F\}}(M, N)=H^j(\Hom_{R\{F\}}(D_\bullet, N))\oplus H^j(\Hom_{R\{F\}}(C_\bullet[-1], N)).
\end{equation}
But for every right $R\{F\}$-module $N$, $\Hom_{R\{F\}}(-\otimes_RR\{F\}, N)\cong \Hom_R(-,N)$. So applying $\Hom_{R\{F\}}(-, N)$ to $D_\bullet$ and $C_\bullet$ are the same as applying $\Hom_R(-, N)$ to \[0\to P_{k}\to P_{k-1}\to\cdots\to P_0\to 0\] and \[0\to P_{k}^{(1)}\to P_{k-1}^{(1)}\to\cdots\to P_0^{(1)}\to 0,\] which are $R$-projective resolutions of $M$ and $M^{(1)}$ respectively. Hence we know that
\begin{equation}\label{cohomology of D and C in terms of Ext^j over R}
H^j(\Hom_{R\{F\}}(D_\bullet, N))\oplus H^j(\Hom_{R\{F\}}(C_\bullet[-1], N))=\Ext^{j}_R(M, N)\oplus\Ext_R^{j-1}(M^{(1)}, N).
\end{equation}
Now (\ref{relating Ext_R[F] with Ext_R}) follows from (\ref{Ext^j in terms of cohomology of D and C}) and (\ref{cohomology of D and C in terms of Ext^j over R}).

In particular, we can take two $R$-modules $M$ and $N$ such that $\Ext_R^d(M^{(1)}, N)\neq 0$ (for example, take $N=R$ and $M=R/(x_1,\dots,x_d)$ where $x_1,\dots, x_d$ is a regular sequence in $R$). Applying (\ref{relating Ext_R[F] with Ext_R}) to $j=d+1$ gives \[\Ext_{R\{F\}}^{d+1}(M, N)=\Ext_R^d(M^{(1)}, N)\neq 0.\] Hence the global dimension of right $R\{F\}$-modules is at least $d+1$. Since we have already shown it is bounded by $d+1$, this completes the proof that the global dimension of right $R\{F\}$-modules is exactly $d+1$.
\end{proof}

We can use the method in the proof of Theorem \ref{right R[F]-module has finite global dimension} to compute some $\Ext$ groups in the category of right $R\{F\}$-modules. Below we give an example which is a key ingredient when we show that the global dimension of $F_R$-module is $d+1$.
\begin{example}
\label{example on injective dimension of omega_R}
Let $R$ be an $F$-finite regular ring of dimension $d$ such that there exists a canonical module $\omega_R$ with $F^!\omega_R\cong\omega_R$. Let $\omega_R^\infty:=\oplus_{j\in \mathbb{Z}}\omega_Rz_j$ be an infinite direct sum of $\omega_R$. We give $\omega_R^\infty$ the right $R\{F\}$-module structure by setting $\tau$: $\omega_R^\infty\to F^!(\omega_R^\infty)\cong \omega_R^\infty$ such that \[\tau(yz_j)=yz_{j+1}\] for every $y\in\omega_R$. It is clear that $\omega_R^\infty$ is in fact a unit right $R\{F\}$-module, and it corresponds to the $F_R$-module $R^\infty$ described in Example \ref{example of F_R-module by shift} under Theorem \ref{equivalence of category of F-mod and unit right mod}.
\end{example}

\begin{lemma} \label{injective dimension of omega_R is d+1}
With the same notations as in Example \ref{example on injective dimension of omega_R}, we have
\begin{equation}
\id_{R\{F\}}\omega_R^\infty=d+1.
\end{equation}
\end{lemma}
\begin{proof}
We first notice that the right $R\{F\}$-module structure on $\omega_R$ defined by $\omega_R\cong F^!\omega_R$ induces a canonical map $\phi$: $\omega_R^{(1)}\to \omega_R$, which is a generator of the free $R^{(1)}$-module $\Hom_R(\omega_R^{(1)}, \omega_R)\cong R^{(1)}$. That is, any map in $\Hom_R(\omega_R^{(1)}, \omega_R)$ can be expressed as $\phi(r^{(1)}\cdot-)$ for some $r^{(1)}\in R^{(1)}$ (we refer to \cite{BlickleBoeckleCartierModulesFiniteness} for more details on this).

Next we fix $x_1,\dots,x_d$ a regular sequence in $R$. We note that \[\widetilde{\phi}:=\phi((x_1^{(1)}\cdots x_d^{(1)})^{p-1}\cdot-)\in\Hom_R(\omega_R^{(1)}, \omega_R)\] satisfies $\widetilde{\phi}((x_1^{(1)},\dots, x_d^{(1)})\omega_R^{(1)})\subseteq (x_1,\dots,x_d)\omega_R$, so it induces a map \[(\omega_R/(x_1,\dots,x_d)\omega_R)^{(1)}\to \omega_R/(x_1,\dots,x_d)\omega_R.\] That is, $\widetilde{\phi}$ gives $\omega_R/(x_1,\dots,x_d)\omega_R$ a right $R\{F\}$-module structure. It is clear that we can lift this map $\widetilde{\phi}$ to the Koszul complex $K_\bullet(x_1,\dots,x_d; \omega_R)$ as follows:
\[\xymatrixcolsep{1pc}\xymatrix{
0 \ar[r]  & \omega_R^{(1)} \ar[r] \ar@{.>}[d]^{\phi} & {\cdots} \ar[r] & (\omega_R^d)^{(1)} \ar[r]  \ar@{.>}[d] & \omega_R^{(1)} \ar[r]\ar[d]^{\widetilde{\phi}} & (\omega_R/(x_1,\dots,x_d)\omega_R)^{(1)} \ar[r] \ar[d]^{\widetilde{\phi}}    & 0\\
0 \ar[r] & \omega_R \ar[r] & {\cdots} \ar[r] & \omega_R^d \ar[r] & \omega_R \ar[r] &  \omega_R/(x_1,\dots,x_d)\omega_R \ar[r]  &0\\
}\]
Chasing through the diagram, one can check that the induced map on the last spot of the above commutative diagram is exactly the map $\phi$ (the generator of $\Hom_R(\omega_R^{(1)}, \omega_R)$).

Now we apply (\ref{huge commutative diagram}) to $M=\omega_R/(x_1,\dots,x_d)\omega_R$ with structure map $\widetilde{\phi}$ and let the first line in (\ref{huge commutative diagram}) be the Koszul complex $K_\bullet(x_1,\dots,x_d; \omega_R)$. The above argument shows that the induced right $R\{F\}$-module structure on $P_d=\omega_R$ is given by the canonical map $\phi$: $\omega_R^{(1)}\to\omega_R$ (i.e., it corresponds to $\omega_R\cong F^!\omega_R$). As in Theorem \ref{right R[F]-module has finite global dimension}, the mapping cone of $C_\bullet\to D_\bullet$ is a right $R\{F\}$-projective resolution of $\omega_R/(x_1,\dots,x_d)\omega_R$ of length $d+1$, and the tail of this resolution is
\begin{equation}
\label{tail of resolution of R/(x_1,...,x_d)}
0\rightarrow \omega_R^{(1)}\otimes_RR\{F\}\xrightarrow{h} \omega_R\otimes_RR\{F\}\oplus (\omega_R^d)^{(1)}\otimes_RR\{F\}\rightarrow \cdots
\end{equation}
where we have
\begin{equation}
\label{map on the tail of mapping cone}
h(y^{(1)}\otimes F^i)=(-1)^d (y\otimes F^{i+1}-\phi(y^{(1)})\otimes F^i)\oplus (x_1^{(1)}y^{(1)},\dots,x_d^{(1)}y^{(1)})\otimes F^i
\end{equation}
for every $y\in \omega_R$. Now we apply $\Hom_{R\{F\}}(-, \omega_R^\infty)$ to (\ref{tail of resolution of R/(x_1,...,x_d)}) and identify $\Hom_R(-, \omega_R^\infty)=\Hom_{R\{F\}}(-\otimes_RR\{F\}, \omega_R^\infty)$, we get
\begin{equation}
\label{dual of the tail of resolution}
0\leftarrow \Hom_R(\omega_R^{(1)}, \omega_R^\infty)\xleftarrow{h^\vee} \Hom_R(\omega_R, \omega_R^\infty) \oplus \Hom_R(\omega_R^{(1)}, \omega_R^\infty)^d \leftarrow\cdots.
\end{equation}
Since $\Hom_R(\omega_R^{(1)},\omega_R)\cong R^{(1)}$ and $\Hom_R(\omega_R, \omega_R)=R$, we can rewrite (\ref{dual of the tail of resolution}) as
\[0\leftarrow \oplus_{j\in \mathbb{Z}}R^{(1)} \xleftarrow{h^\vee} (\oplus_{j\in \mathbb{Z}}R) \oplus (\oplus_{j\in \mathbb{Z}}R^{(1)})^d \leftarrow\cdots.\]
And after a careful computation using (\ref{map on the tail of mapping cone}) and the right $R\{F\}$-module structure of $\omega_R^\infty$, we have
\begin{equation}
\label{image of the dual of h}
h^\vee(\{s_j\}\oplus (\{t_{1j}^{(1)}\},\dots,\{t_{dj}^{(1)}\}))=\{(-1)^d((s_j^{(1)})^p-s_{j-1}^{(1)})+\sum_{i=1}^d x_i^{(1)}t_{ij}^{(1)}\}_{j\in \mathbb{Z}}
\end{equation}
where $\{s_j\}$ denotes an element in $\oplus_{j\in \mathbb{Z}}R$ and $(\{t_{1j}^{(1)}\},\dots,\{t_{dj}^{(1)}\})$ denotes an element in $(\oplus_{j\in \mathbb{Z}}R^{(1)})^d$. The key point here is that $h^\vee$ is {\it not} surjective. To be more precise, I claim $(-1)^{d-1}z_0=(\dots,0,(-1)^{d-1},0,0,\dots)$ (i.e., the element in $\oplus_{j\in \mathbb{Z}}R^{(1)}$ with $0$-th entry $(-1)^{d-1}$ and other entries $0$) is not in the image of $h^\vee$. This is because $\sum_{i=1}^d x_i^{(1)}t_{ij}^{(1)}$ can only take values in $\oplus_{j\in\mathbb{Z}}(x_1^{(1)},\dots,x_d^{(1)})$, so if $z_0\in\im h^\vee$, then mod $\oplus_{j\in\mathbb{Z}}(x_1^{(1)},\dots,x_d^{(1)})$, we know by (\ref{image of the dual of h}) that $(\overline{s_j}^{(1)})^p-\overline{s_{j-1}}^{(1)}=0$ for $j\neq 0$ and $(\overline{s_0}^{(1)})^p-\overline{s_{-1}}^{(1)}=-1$. And it is straightforward to see that a solution $\{s_j\}_{j\in\mathbb{Z}}$ to this system must satisfy $\overline{s_j}=0$ when $j\geq 0$ and $\overline{s_j}=\overline{1}$ when $j<0$ where $\overline{s}$ denotes the image of $s\in R$ mod $(x_1,\dots,x_d)$. So there is no solution in $\oplus_{j\in \mathbb{Z}}R$, since $\overline{s_j}=\overline{1}$ for every $j<0$ implies there has to be infinitely many nonzero $s_j$.

Hence we get
\begin{equation}
\label{ext^d+1 of omega_R}
\Ext_{R\{F\}}^{d+1}(\omega_R/(x_1,\dots,x_d)\omega_R, \omega_R^\infty)\cong \coker {h^\vee} \neq 0.
\end{equation}
Combining (\ref{ext^d+1 of omega_R}) with Theorem \ref{right R[F]-module has finite global dimension} completes the proof of the Lemma.
\end{proof}

\begin{remark}
One might hope that $\id_{R\{F\}}\omega_R=d+1$ by the same type computation used in Lemma \ref{injective dimension of omega_R is d+1}. But there is a small gap when doing this. The problem is, when we apply $\Hom_{R\{F\}}(-, \omega_R)$ to (\ref{tail of resolution of R/(x_1,...,x_d)}) and compute $\coker h^\vee$, we get
\begin{equation}
\label{ext of omega_R in the remark}
\Ext_{R\{F\}}^{d+1}(\omega_R/(x_1,\dots,x_d)\omega_R, \omega_R)=\coker{h^\vee}\cong\D\frac{R}{(x_1,\dots,x_d)+\{r^p-r\}_{r\in R}}.
\end{equation}
So if the set $\{r^p-r\}_{r\in R}$ can take all values of $R$ (this happens, for example when $(R,\m)$ is a complete regular local ring with algebraically closed residue field, see Remark \ref{finite conditions on Ext^1(R, R)}), then $\Ext_{R\{F\}}^{d+1}(\omega_R/(x_1,\dots,x_d)\omega_R, \omega_R)=0$. So we cannot get the desired result in this way. However, we do get from (\ref{ext of omega_R in the remark}) that if $(R,\m, K)$ is an $F$-finite regular local ring with $K\cong R/\m$ a {\it finite} field, then $\id_{R\{F\}}R=d+1$.
\end{remark}

Now we prove our main result. We start by proving that the $\Ext$ groups are the same no matter one compute in the category of unit right $R\{F\}$-modules or the category of right $R\{F\}$-modules. We give two proofs of this result, the second proof is suggested by the referee, which in fact shows a stronger result.
\begin{theorem}
\label{ext groups are the same}
Let $R$ be an $F$-finite regular ring of dimension $d$ such that there exists a canonical module $\omega_R$ with $F^!\omega_R\cong\omega_R$. Let $M$, $N$ be two unit right $R\{F\}$-modules. Then we have $\Ext_{uR\{F\}}^i(M,N)\cong\Ext_{R\{F\}}^i(M,N)$ for every $i$. In particular, the category of unit right $R\{F\}$-modules and the category of $F_R$-modules has finite global dimension $\leq d+1$.
\end{theorem}
\begin{proof}
First we note that by Theorem \ref{right R[F]-module has finite global dimension} and Theorem \ref{equivalence of category of F-mod and unit right mod}, it is clear that we only need to show $\Ext_{uR\{F\}}^i(M,N)\cong\Ext_{R\{F\}}^i(M,N)$ for $M$, $N$ two unit right $R\{F\}$-modules. Below we give two proofs of this fact.

\vspace{1em}

{\parindent 0pt{\it First proof}}: We use Yoneda's characterization of $\Ext^i$ ({\it cf.} Chapter 3.4 in \cite{WeibelHomological}). Note that this is the same as the derived functor $\Ext^i$ whenever the abelian category has enough injectives or enough projectives, hence holds for both the category of unit right $R\{F\}$-modules and the category of right $R\{F\}$-modules (unit right $R\{F\}$-modules has enough injectives by Theorem \ref{Hochster existence of enough injectives} and Theorem \ref{equivalence of category of F-mod and unit right mod}). An element in $\Ext_{uR\{F\}}^i(M,N)$ (resp. $\Ext_{R\{F\}}^i(M,N)$) is an equivalence class of exact sequences of the form \[\xi: 0\rightarrow N\rightarrow X_1\rightarrow \cdots\rightarrow X_i\rightarrow M\rightarrow 0\] where each $X_i$ is a unit right $R\{F\}$-module (resp. right $R\{F\}$-module) and the maps are maps of unit right $R\{F\}$-modules (resp. maps of right $R\{F\}$-modules). The equivalence relation is generated by the relation $\xi_X\thicksim\xi_Y$ if there is a commutative diagram
\[ \xymatrix{
0 \ar[r]  & N \ar[r] \ar[d]^{\cong} & X_1\ar[r] \ar[d]  &{\cdots} \ar[r] &X_i \ar[r] \ar[d]  &M \ar[r] \ar[d]^{\cong}   &0\\
0 \ar[r] &N  \ar[r] & Y_1 \ar[r]  &{\cdots} \ar[r] & Y_i  \ar[r] & M \ar[r] &0
}\]

From this characterization of $\Ext^i$ it is clear that we have a well-defined map \[\iota: \Ext_{uR\{F\}}^i(M,N)\rightarrow\Ext_{R\{F\}}^i(M,N)\] taking an equivalence class of an exact sequence of unit right $R\{F\}$-modules to the same exact sequence but viewed as an exact sequence in the category of right $R\{F\}$-modules.

Conversely, if we have an element in $\Ext_{R\{F\}}^i(M,N)$, say $\xi$, we have an exact sequence of right $R\{F\}$-modules, this induces a commutative diagram
\[ \xymatrix{
0 \ar[r]  & N \ar[r] \ar[d]^{\cong} & X_1\ar[r] \ar[d]  &{\cdots} \ar[r] &X_i \ar[r] \ar[d]  &M \ar[r] \ar[d]^{\cong}   &0\\
0 \ar[r] &F^!(N)  \ar[r] \ar[d]^{\cong}& F^!(X_1) \ar[r] \ar[d] &{\cdots} \ar[r] & F^!(X_i)  \ar[r] \ar[d] & F^!(M) \ar[r] \ar[d]^{\cong} &0\\
0 \ar[r] &(F^!)^2(N)  \ar[r] \ar[d]^{\cong} & (F^!)^2(X_1) \ar[r] \ar[d] &{\cdots} \ar[r] & (F^!)^2(X_i)  \ar[r] \ar[d] & (F^!)^2(M) \ar[r]\ar[d]^{\cong} &0\\
 { } &{ }&{} &{}&{}&{}
}\]
Taking direct limits for columns and noticing that $M$, $N$ are unit right $R\{F\}$-modules, we get a commutative diagram
\begin{equation}
\label{xi is equivalent to xi'}
\xymatrix{
0 \ar[r]  & N \ar[r] \ar[d]^{\cong} & X_1\ar[r] \ar[d]  &{\cdots} \ar[r] &X_i \ar[r] \ar[d]  &M \ar[r] \ar[d]^{\cong}   &0\\
0 \ar[r] &N  \ar[r] & \varinjlim{(F^!)^e(X_1)} \ar[r]  &{\cdots} \ar[r] & \varinjlim{(F^!)^e(X_i)}  \ar[r] & M \ar[r] &0
}
\end{equation}
Since the functor $F^!(-)$ and the direct limit functor are both exact, the bottom sequence is still exact and hence it represents an element in $\Ext_{uR\{F\}}^i(M,N)$ (note that each $\varinjlim{(F^!)^e(X_j)}$ is a unit right $R\{F\}$-module by Remark \ref{generating morphism for unit right R[F]-mod}). We call this element $\xi'$. Then we have a map \[\eta: \Ext_{R\{F\}}^i(M,N)\xrightarrow{\xi\mapsto\xi'} \Ext_{uR\{F\}}^i(M,N).\] This map is well-defined because it is easy to check that if $\xi_1\sim\xi_2$, then we also have $\xi_1'\sim\xi_2'$. It is also straightforward to check that $\iota$ and $\eta$ are inverses of each other. Obviously $\eta\circ\iota([\xi])=[\xi]$ and $\iota\circ\eta([\xi])=[\xi']=[\xi]$, where the last equality is by (\ref{xi is equivalent to xi'}) (which shows that $\xi\sim\xi'$, and hence they represent the same equivalence class in $\Ext_{R\{F\}}^i(M,N)$).

\vspace{1em}
{\parindent 0pt{\it Second proof}}: By Theorem \ref{Hochster existence of enough injectives} and Theorem \ref{equivalence of category of F-mod and unit right mod} we know that the category of unit right $R\{F\}$-module has enough injectives. Now we show that every injective object in the category of unit right $R\{F\}$ modules is in fact injective in the category of right $R\{F\}$-modules. To see this, let $I$ be a unit right $R\{F\}$-injective module. It is enough to show that whenever we have $0\rightarrow I\rightarrow W$ for some right $R\{F\}$-module $W$, the sequence splits. But $0\rightarrow I\rightarrow W$ induces the following diagram:
\[ \xymatrix{
0 \ar[r]  & I \ar[r] \ar[d]^{\cong} & W \ar[d] \\
0 \ar[r] &F^!(I)  \ar[r] \ar[d]^{\cong}& F^!(W) \ar[d] \\
0 \ar[r] &(F^!)^2(I)  \ar[r] \ar[d]^{\cong} & (F^!)^2(W)  \ar[d]\\
 { } &{ }&{}
}\]
Taking direct limit for the columns we get
\begin{equation}
\label{commutative diagram when proving injective objects}
\xymatrix{
0 \ar[r]  & I\ar[r] \ar[d]^{\cong} & W \ar[d]  \\
0 \ar[r] &I  \ar[r] & \varinjlim{(F^!)^e(W)}
}
\end{equation}

We still have exactness because the functor $F^{!}(-)$ and the direct limit functor are both exact. We also note that $\varinjlim{(F^!)^e(W)}$ is a unit right $R\{F\}$-module by Remark \ref{generating morphism for unit right R[F]-mod}. Now since $I$ is injective in the category of unit right $R\{F\}$-modules, we know that the bottom map $0\to I\rightarrow \varinjlim{(F^!)^e(W)}$ splits as unit right $R\{F\}$-modules, so it also splits as right $R\{F\}$-modules. But now composing with the commutative diagram (\ref{commutative diagram when proving injective objects}) shows that the map $0\to I\rightarrow W$ splits as right $R\{F\}$-modules.

Now it is clear that $\Ext_{uR\{F\}}^i(M,N)\cong\Ext_{R\{F\}}^i(M,N)$. Because one can take an injective resolution of $N$ in the category of unit right $R\{F\}$-modules:
\begin{equation}
\label{injective resolution of a unit right module}
0\rightarrow N \rightarrow I_0\rightarrow I_1\rightarrow\cdots\cdots.
\end{equation}
By the above argument this can be also viewed as an injective resolution in the category of right $R\{F\}$-modules. Since applying $\Hom_{R\{F\}}(M, -)$ and $\Hom_{uR\{F\}}(M, -)$ to (\ref{injective resolution of a unit right module}) are obviously the same, we know that $\Ext_{uR\{F\}}^i(M,N)\cong\Ext_{R\{F\}}^i(M,N)$.
\end{proof}

\begin{theorem}
\label{the category of F-mod has finite global dimension d+1}
Let $R$ be an $F$-finite regular ring of dimension $d$ such that there exists a canonical module $\omega_R$ with $F^!\omega_R\cong\omega_R$. Then the category of unit right $R\{F\}$-modules and the category of Lyubeznik's $F_R$-modules both have finite global dimension $d+1$.
\end{theorem}
\begin{proof}
By Theorem \ref{equivalence of category of F-mod and unit right mod}, it suffices to show that the category of unit right $R\{F\}$-modules has finite global dimension $d+1$. By Theorem \ref{ext groups are the same}, we know that the global dimension is at most $d+1$.

Now let $\omega_R^\infty$ be the unit right $R\{F\}$-module described in Example \ref{example on injective dimension of omega_R}. If the global dimension is $\leq d$, then we know that $\omega_R^\infty$ has a unit right $R\{F\}$-injective resolution of length $d'\leq d$:
\begin{equation}
\label{injective resolution of length d}
0\rightarrow \omega_R^\infty\rightarrow I_0\rightarrow I_1\rightarrow\cdots\rightarrow I_{d'}\rightarrow 0
\end{equation}
But by the argument in the second proof of Theorem \ref{ext groups are the same}, we know that each $I_j$ is injective in the category of right $R\{F\}$-modules. So (\ref{injective resolution of length d}) can be viewed as an injective resolution of $\omega_R^\infty$ in the category of right $R\{F\}$-modules. And hence $\id_{R\{F\}}\omega_R^\infty\leq d$, which contradicts Lemma \ref{injective dimension of omega_R is d+1}.
\end{proof}

\begin{remark}
It is clear from Theorem \ref{equivalence of category of F-mod and unit right mod} and the above proof of Theorem \ref{the category of F-mod has finite global dimension d+1} that \[\id_{F_R}R^\infty=\id_{uR\{F\}}\omega_R^\infty=d+1.\]
\end{remark}

\section{non-Finiteness of $\Ext_{F_R}^1$}
In this section we study the group  $\Ext_{F_R}^1(M,N)$ when $M$, $N$ are $F_R$-finite $F_R$-modules. We prove that when $(R,\m, K)$ is a regular local ring, $\Ext^1_{F_R}(M,N)$ is finite when $K=R/\m$ is separably closed and $M$ is supported only at $\m$. However, we provide examples to show that in general $\Ext^1_{F_R}(M,N)$ is not necessarily a finite set. We begin with some lemmas.
\begin{lemma}[{\it cf.} Proposition 3.1 in \cite{LyubeznikFModulesApplicationsToLocalCohomology}]
\label{correspondence between F_R and F_S-modules}
Let $S$ be a regular ring of characteristic $p>0$ and let $R\rightarrow S$ be a surjective homomorphism with kernel $I\subseteq R$. There exists an equivalence of categories between $F_R$-modules supported on $\Spec S=V(I)\subseteq \Spec R$ and $F_S$-modules. Under this equivalence the $F_R$-finite $F_R$-modules supported on $\Spec S=V(I)\subseteq \Spec R$ correspond to the $F_S$-finite $F_S$-modules.
\end{lemma}

\begin{lemma}[{\it cf.} Theorem 4.2(c)(e) in \cite{HochsterfinitenesspropertyofLyubeznikFmodule}]
\label{F_K-modules when K is a separably closed field} Let $K$ be a separably closed field. Then every $F_K$-finite $F_K$-module is isomorphic with a finite direct sum of copies of $K$ with the standard $F_K$-module structure. Moreover, $\Ext_{F_K}^1(K,K)=0$.
\end{lemma}
\begin{lemma}
\label{structure of F_R-modules supported at m}
Let $(R,\m,K)$ be a regular local ring with $K$ separably closed. Then every $F_R$-finite $F_R$-module supported only at $\m$ is isomorphic (as an $F_R$-module) with a finite direct sum of copies of $E=E(R/\m)$ (where $E$ is equipped with the standard $F_R$-module structure). Moreover, $\Ext_{F_R}^1(E,E)=0$.
\end{lemma}
\begin{proof}
This is clear from Lemma \ref{correspondence between F_R and F_S-modules} (applied to $S=K$ and $I=\m$) and Lemma \ref{F_K-modules when K is a separably closed field} because it is straightforward to check that the standard $F_R$-module structure on $E$ corresponds to the standard $F_K$-module structure on $K$ via Lemma \ref{correspondence between F_R and F_S-modules}.
\end{proof}

\begin{theorem}
Let $(R,\m, K)$ be a regular local ring such that $K$ is separably closed and let $M$, $N$ be $F_R$-finite $F_R$-modules. Then $\Ext_{F_R}^1(M,N)$ is finite if $M$ is supported only at $\m$.
\end{theorem}
\begin{proof}
Since $K$ is separably closed, by Lemma \ref{structure of F_R-modules supported at m} we know that $M$ is a finite direct sum of copies of $E$ in the category of $F_R$-modules. So it suffices to show that $\Ext_{F_R}^1(E, N)$ is finite. For every exact sequence of $F_R$-finite $F_R$-modules \[0\rightarrow N_1\rightarrow N_2\rightarrow N_3\rightarrow 0,\] the long exact sequence of $\Ext$ gives \[\Ext_{F_R}^1(E,N_1)\rightarrow\Ext_{F_R}^1(E,N_2)\rightarrow\Ext_{F_R}^1(E,N_3).\] So we immediately reduce to the case that $N$ is simple (since $R$ is local, every $F_R$-finite $F_R$-module has finite length by Theorem 3.2 in \cite{LyubeznikFModulesApplicationsToLocalCohomology}).

We want to show that $\Ext_{F_R}^1(E,N)$ is finite when $N$ is simple. There are two cases: $\Ass_R(N)=\m$ or $\Ass_R(N)=P\neq\m$ (by Theorem 2.12(b) in \cite{LyubeznikFModulesApplicationsToLocalCohomology}). If $\Ass_R(N)=\m$, then $N\cong E$ as $F_R$-modules by Lemma \ref{structure of F_R-modules supported at m}. So $\Ext_{F_R}^1(E, N)=\Ext_{F_R}^1(E,E)=0$ by Lemma \ref{structure of F_R-modules supported at m}.

If $\Ass_R(N)=P\neq\m$, by Yoneda's characterization of $\Ext$ groups, it suffices to show that we only have a finite number of isomorphism classes of short exact sequences \[0\rightarrow N\rightarrow L\rightarrow E\rightarrow 0\] of $F_R$-modules. We first show the number of choices of isomorphism classes for $L$ is finite. Say $\Ass_R(N)={P}\neq \m$, we have $P \in \Ass_R(L) \subseteq \{ P,\m \}$. If $\Ass_R(L)=\{ P,\m \}$, then $H_\m^0(L)\neq 0$ and it does not intersect $N$. So $H_\m^0(L)\oplus N$ is an $F_R$-submodule of $L$. Hence we must have $L\cong H_\m^0(L)\oplus N \cong E\oplus N$ since $L$ has length 2 as an $F_R$-module. If $\Ass_R(L)=\{P\}$, we can pick $x\in \m-P$. Localizing at $x$ gives a short exact sequence \[0\rightarrow N_x\rightarrow  L_x\rightarrow E_x\rightarrow 0.\] But $E_x=0$, so we get $N_x\cong L_x$ as $F_R$-module. Since $x$ is not in $P$, we have $L\hookrightarrow L_x$ as $F_R$-module. That is, $L$ is isomorphic to an $F_R$-submodule of $L_x$, hence is isomorphic to an $F_R$-submodule of $N_x$. But $N_x$ is $F_R$-finite by Proposition 2.9(b) in \cite{LyubeznikFModulesApplicationsToLocalCohomology}, so it only has finitely many $F_R$-submodules by Theorem \ref{Hochster finiteness of Hom(M,N)}. This proves that the number of choices of isomorphism classes for $L$ is finite.

Because the number of choices of isomorphism classes for $L$ is finite, and for each $F_R$-finite $F_R$-module $L$, $\Hom_{F_R}(N, L)$ is always finite by Theorem \ref{Hochster finiteness of Hom(M,N)}. It follows that the number of isomorphism classes of short exact sequences $0\rightarrow N\rightarrow L\rightarrow E\rightarrow 0$ is finite.
\end{proof}

If $M$ is an $F_R$-module with structure morphism $\theta_M$, for every $x\in M$ we use $x^p$ to denote $\theta_M^{-1}(1\otimes x)$. Notice that when $M=R$ with the standard $F_R$-module structure, this is exactly the usual meaning of $x^p$. We let $G_M$ denote the set $\{x^p-x|x\in M\}$.  It is clear that $G_M$ is an abelian subgroup of $M$.

\begin{theorem}
\label{structure of Ext(R,M)}
Let $R$ be a regular ring. Giving $R$ the standard $F_R$-module structure, we have $\Ext_{F_R}^1(R, M)\cong M/G_M$ as an abelian group for every $F_R$-module $M$.
\end{theorem}
\begin{proof}
By Yoneda's characterization of $\Ext$ groups, an element in $\Ext_{F_R}^1(R, M)$ can be represented by an exact sequence of $F_R$-modules \[0\rightarrow M\rightarrow L\rightarrow R\rightarrow 0.\] It is clear that $L\cong M\oplus R$ as $R$-module. Moreover, one can check that the structure isomorphism $\theta_L$ composed with $\theta_M^{-1}\oplus\theta_R^{-1}$ defines an isomorphism \[M\oplus R\xrightarrow{\theta_L} F(M)\oplus F(R)\xrightarrow{\theta_M^{-1}\oplus\theta_R^{-1}} M\oplus R\] which sends $(y, r)$ to $(y+rz, r)$ for every $(y, r)\in M\oplus R$ and for some $z\in M$. Hence, giving a structure isomorphism of $L$ is equivalent to giving some $z\in M$. Therefore, $\theta_L$ is determined by an element $z\in M$. Two exact sequences with structure isomorphism $\theta_L$, $\theta'_L$ are in the same isomorphism class if and only if there exists a map $g$: $L\rightarrow L$, sending $(y, r)$ to $(y+rx, r)$ for some $x\in M$ such that \[(1\otimes g)\circ\theta_L=\theta'_L\circ g .\]  Now we apply $\theta_M^{-1}\oplus\theta_R^{-1}$ on both side. If $\theta_L$, $\theta'_L$ are determined by $z_1$ and $z_2$ respectively, a direct computation gives that \[(\theta_M^{-1}\oplus\theta_R^{-1})\circ(1\otimes g)\circ\theta_L(y,r)=(y+rz_1+rx^p, r)\] while \[(\theta_M^{-1}\oplus\theta_R^{-1})\circ\theta'_L\circ g(y,r)=(y+rz_2+rx, r).\] So $\theta_L$ and $\theta'_L$ are in the same isomorphism class if and only if there exists $x\in M$ such that \[z_2-z_1=x^p-x .\] So $\Ext_{F_R}^1(R, M)\cong M/G_M$ as an abelian group.
\end{proof}

Before we use Theorem \ref{structure of Ext(R,M)} to study examples, we make the following remark. I would like to thank the referee for his suggestions on this remark.
\begin{remark}
\label{finite conditions on Ext^1(R, R)}
\begin{enumerate}
\item In the case that $R$ is a regular ring which is {\it $F$-finite} and {\it local}. By Theorem \ref{equivalence of category of F-mod and unit right mod}, we can identify the category of $F_R$-modules with the category of unit right $R\{F\}$-modules ($\omega_R=R$ is unique). And by Theorem \ref{ext groups are the same}, we can compute $\Ext^1_{F_R}(R, M)\cong\Ext^1_{uR\{F\}}(R, M)\cong\Ext^1_{R\{F\}}(R, M)$ by taking the right $R\{F\}$-projective resolution of $R$ and then applying $\Hom_{R\{F\}}(-, M)$. Note that one right $R\{F\}$-projective resolution of $R$ is given by \[0\rightarrow R^{(1)}\otimes_RR\{F\}\rightarrow R\{F\}\rightarrow R\rightarrow 0\] as in Lemma \ref{two-step resolution of right R[F]-module}. Thus in this case one can give another proof of Theorem \ref{structure of Ext(R,M)}, we leave the details to the reader.
\item When $(R,\m)$ is a strict Henselian local ring (e.g., $(R,\m)$ is a complete local ring with algebraically closed residue field), the Artin-Schreier sequence \[0\rightarrow \mathbb{F}_p\rightarrow R\xrightarrow{x^p-x} R\rightarrow 0\] is exact in the Zariski topology, which shows that $G_R=R$, and hence $\Ext_{F_R}^1(R, R)=0$ when $R$ is a strict Henselian local ring. In particular, applying this to $R=K$ a separably closed field, we recover Lemma \ref{F_K-modules when K is a separably closed field}.
\end{enumerate}
\end{remark}

Now we give some examples to show that, in general, $\Ext_{F_R}^1(R, M)\cong M/G_M$ is not necessarily finite, even in simple cases.
\begin{example}
Let $R=k(t)$ or $k[t]_{(t)}$ with $k$ an algebraically closed field. We will prove that $\Ext_{F_R}^1(R, R)$ is infinite in both cases. By Theorem \ref{structure of Ext(R,M)}, it suffices to show that for $a, b\in k$ ($a, b\neq 0$ in the second case), $\D\frac{1}{t-a}$ and $\D\frac{1}{t-b}$ are different in $R/G_R$ whenever $a\neq b$. Otherwise there exists $\D\frac{h(t)}{g(t)}\in R$ with $h(t), g(t)\in k[t]$ ($g(t)$ is not divisible by $t$ in the second case) and $\gcd(h(t), g(t))=1$ such that
\begin{equation*}
\frac{1}{t-a}-\frac{1}{t-b}=\frac{h(t)^p}{g(t)^p}-\frac{h(t)}{g(t)}
\end{equation*}
which gives
\begin{equation}
\label{equation in the example of non-finiteness of ext}
\frac{a-b}{t^2-(a+b)t+ab}=\frac{h(t)^p-h(t)\cdot g(t)^{p-1}}{g(t)^p}.
\end{equation}
Since $\gcd(h(t), g(t))=1$, $\gcd(h(t)^p-h(t)\cdot g(t)^{p-1}, g(t)^p)=1$. So from (\ref{equation in the example of non-finiteness of ext}) we know that $g(t)^p|(t^2-(a+b)t+ab)$. This is clearly impossible.
\end{example}

\begin{example}
\label{E not injective as F_R-module}
Let $(R,\m, K)$ be a regular local ring of dimension $d\geq 1$. Let $E=E(R/\m)$ be the injective hull of the residue field. We will show that $\Ext_{F_R}^1(R, E)$ is not zero and is in fact infinite when $K$ is infinite. In particular, $E=E(R/\m)$, though injective as an $R$-module, is {\it not} injective as an $F_R$-module (with its standard $F_R$-structure) when $\dim R\geq 1$.

Recall that $E=\D\varinjlim_n\frac{R}{(x_1^n,\dots,x_d^n)}$. So every element $z$ in $E$ can be expressed as $(r; x_1^n,\dots,x_d^n)$ for some $n\geq 1$ (which means $z$ is the image of $r$ in the $n$-th piece in this direct limit system). By Theorem \ref{structure of Ext(R,M)}, $\Ext_{F_R}^1(R, E)\cong E/G_E$. I claim that any two different socle elements $u_1$, $u_2$ are different in $E/G_E$. Suppose this is not true, we have:
\begin{equation}
\label{relation on z and u}
u_1-u_2=z^p-z
\end{equation}
in $E$. Since $u_1-u_2$ is a nonzero element in the socle of $E$, we may write $u_1-u_2=(\lambda; x_1,\dots,x_d)$ for some $\lambda\neq 0$ in $K$. Say $z=(r; x_1^n,\dots,x_d^n)$ with $n$ minimum. Then (\ref{relation on z and u}) will give
\[(r; x_1^n,\dots,x_d^n)=(\lambda; x_1,\dots,x_d)+(r^p; x_1^{np},\dots,x_d^{np}).\]
This will give us
\begin{equation}
\label{equation on r}
r^p+\lambda(x_1\cdots x_d)^{np-1}-r (x_1\cdots x_d)^{np-n}\in (x_1^{np},\dots,x_d^{np}).
\end{equation}

If $n=1$, then $0\neq z\in \soc(E)$, hence $r$ is a nonzero unit in $R$. But (\ref{equation on r}) shows that $r^p\in (x_1,\dots,x_d)$ which is a contradiction.

If $n\geq2$, we have $np-1\geq np-n\geq p$. We know from (\ref{equation on r}) that for every $1\leq i\leq d$, we have $r^p\in (x_1^{np},\dots,x_{i-1}^{np}, x_i^{p},x_{i+1}^{np}, \dots,x_d^{np})$.  Hence $r\in (x_1^n,\dots,x_{i-1}^n, x_i, x_{i+1}^n,\dots, x_d^n)$ for every $1\leq i\leq d$. Taking their intersection, we get that $r\in (x_1\cdots x_d, x_1^n,\dots,x_d^n)$. That is, mod $(x_1^n,\dots,x_d^n)$, we have $r=(x_1\cdots x_d)r_0$. But then we have $z=(r_0; x_1^{n-1},\dots, x_d^{n-1})$ contradicting our choice of $n$.

Therefore we have proved that any two different socle elements $u_1$, $u_2$ are different in $\Ext_{F_R}^1(R, E)=E/G_E$. This shows that $\Ext_{F_R}^1(R, E)\neq 0$ and is infinite when $K$ is infinite.
\end{example}

\section*{Acknowledgement} I would like to thank Mel Hochster and Gennady Lyubeznik for reading a preliminary version of the paper and for their helpful and valuable comments. I would like to thank David Speyer for some helpful discussions on Yoneda's $\Ext$ groups. I am grateful to Vasudevan Srinivas for some helpful discussions on the global dimension of $F$-modules. Finally, I would like to express my deep gratitude to the anonymous referee whose valuable suggestions improved the paper considerably. In particular, I would like to thank the referee for suggesting the definition of unit right $R\{F\}$-modules and for pointing out a different proof of Theorem \ref{ext groups are the same} which leads to Theorem \ref{the category of F-mod has finite global dimension d+1} and globalizes the original results, and also for his comments on Remark \ref{finite conditions on Ext^1(R, R)}.

\bibliographystyle{skalpha}
\bibliography{CommonBib}
\end{document}